\documentclass[10pt, a4paper]{article}
\usepackage{mathrsfs, mathtools, amsthm, amssymb, thm-restate}  
\allowdisplaybreaks[4]
\usepackage{paralist, tablists}  

\usepackage{graphicx, xcolor}  
\usepackage{caption}  
\usepackage[labelformat=simple]{subcaption}  

\usepackage[left=25mm, top=25mm, bottom=25mm, right=25mm]{geometry}  
\usepackage[T1]{fontenc} 

\usepackage[inline]{enumitem}  
\usepackage[square, numbers, sort&compress]{natbib}  
\usepackage[
  bookmarks=true,                
  bookmarksnumbered=true,        
  bookmarksopen=true,            
  plainpages=false,              
  colorlinks=true,               
  linkcolor=blue,                
  citecolor=black,               
  anchorcolor=green,             
  urlcolor=blue,                 
  hyperindex=true                
]{hyperref} 

\newtheorem{theorem}{Theorem}[section] 
\newtheorem{lemma}[theorem]{Lemma}

\newtheorem{case}{Case}
\newtheorem{subcase}{Subcase}[case]

\theoremstyle{definition}

\makeatletter
\def\th@plain{%
  \upshape 
}
\makeatother

\newcommand{\ie}{i.e.,\ }

\usepackage{array}
\usepackage{bm}  

\makeatletter
\renewenvironment{proof}[1][\proofname]{\par
  \pushQED{\qed}%
  \normalfont \topsep6\p@\@plus6\p@\relax
  \trivlist
  \item[\hskip\labelsep
        \bfseries
    #1\@addpunct{.}]\ignorespaces
}{%
  \popQED\endtrivlist\@endpefalse
}
\makeatother

\usepackage[capitalise]{cleveref}  
\crefname{claim}{Claim}{Claims}

\begin{document}
\title{On spectral conditions for fractional $k$-extendable graphs}
\author{
    Xiyan Bai\thanks{School of Mathematics and Statistics, Henan University, Kaifeng, 475004, P. R. China.}
    \and
    Tao Wang\thanks{Center for Applied Mathematics, Henan University, Kaifeng, 475004, P. R. China.}
    \and
    Mengke Yang\thanks{School of Mathematics and Statistics, Henan University, Kaifeng, 475004, P. R. China.}
    \and
    Xiaojing Yang\thanks{School of Mathematics and Statistics, Henan University, Kaifeng, 475004, P. R. China.}
}
\date{}
\maketitle

\begin{abstract}
A fractional matching of a graph $G$ is a function $h: E(G) \to [0,1]$ such that $\sum_{e \in E_G(v)} h(e) \leq 1$ for every vertex $v \in V(G)$, where $E_G(v)$ is the set of edges incident to $v$. If $\sum_{e \in E_G(v)} h(e) = 1$ for all $v$, then $h$ is a fractional perfect matching. A graph $G$ is fractional $k$-extendable if it has a matching of size $k$ and every $k$-matching $M$ in $G$ is contained in a fractional perfect matching $h$ such that $h(e)=1$ for every $e \in M$. In this paper, we establish new sufficient conditions for a graph with minimum degree $\delta$ to be fractional $k$-extendable. Our main results provide spectral guarantees for this property based on the distance spectral radius and the signless Laplacian spectral radius.

\textbf{Keywords:} fractional $k$-extendable graph; signless Laplacian spectral radius; distance spectral radius; minimum degree

\textbf{MSC2020: 05C50}
\end{abstract}

\section{Introduction}
Let $G$ be a connected graph. The \emph{spectral radius} $\rho(G)$ is the largest eigenvalue of its adjacency matrix $A(G)$. The \emph{signless Laplacian matrix} of $G$ is defined as $Q(G) = D(G) + A(G)$, where $D(G)$ is the diagonal degree matrix of $G$; its largest eigenvalue, $q(G)$, is the \emph{signless Laplacian spectral radius}. The \emph{distance matrix} $\mathcal{D}(G) = (d_{i,j})$ is defined by letting $d_{i,j}$ be the length of a shortest path between vertices $v_i$ and $v_j$. The eigenvalues of $\mathcal{D}(G)$, denoted by $\mu_1 \geq \mu_2 \geq \dots \geq \mu_n$, are the \emph{distance eigenvalues} of $G$, and the largest one, $\mu_1$, is the \emph{distance spectral radius}, denoted $\mu(G)$. Let $K_n$ denote the complete graph on $n$ vertices. For two vertex-disjoint graphs $G_1$ and $G_2$, their disjoint union is denoted by $G_1 + G_2$, and their \emph{join}, denoted $G_1 \vee G_2$, is the graph obtained from $G_1 + G_2$ by adding every possible edge between $V(G_1)$ and $V(G_2)$.

A set $M \subseteq E(G)$ is a \emph{matching} if no two edges in $M$ share a vertex. A matching of size $k$ is a \emph{$k$-matching}, and a matching that covers all vertices of $G$ is a \emph{perfect matching} (or a \emph{1-factor}). A graph $G$ with at least $2k+2$ vertices is \emph{$k$-extendable} if it contains a $k$-matching and every $k$-matching can be extended to a perfect matching.

A \emph{fractional matching} is a function $h: E(G) \to [0, 1]$ such that $\sum_{e \in E_G(v)} h(e) \leq 1$ for every vertex $v \in V(G)$, where $E_G(v)$ is the set of edges incident with $v$. If $\sum_{e \in E_G(v)} h(e) = 1$ for all $v \in V(G)$, then $h$ is a \emph{fractional perfect matching}. A graph $G$ with at least $2k+2$ vertices is \emph{fractional $k$-extendable} if every $k$-matching $M$ is contained in some fractional perfect matching $h$ with $h(e) = 1$ for all $e \in M$.

The study of (fractional) $k$-extendable graphs has attracted considerable attention. The concept was introduced by Plummer \cite{MR583220}. Ananchuen and Caccetta \cite{MR1452932} established a necessary minimum degree condition and characterized its realizable values. Robertshaw and Woodall \cite{MR1892694} found a sharp lower bound on the binding number ensuring $k$-extendability. Lou and Yu \cite{MR2043926} characterized the largest $k$-extendable graphs of a given order. Cioab\v{a}, Koolen, and Li \cite{MR3621737} proved that all distance-regular graphs with diameter at least $3$ are $2$-extendable and derived improved bounds on extendability for valency $k \geq 3$. Ma and Liu \cite{MR2102027} characterized fractional $k$-extendable graphs. Other related results include work on binding numbers and fractional matchings \cite{MR2347527}, $A_\alpha$-spectral radius conditions \cite{MR4822576}, fractional $n$-factor-critical graphs \cite{MR2345432}, and a spectral radius condition for fractional $k$-extendability \cite{MR5004976}.

\begin{theorem}[\cite{MR4822576}]
Let $\alpha \in [0, 1)$, and let $G$ be a connected graph of order $n$ with $n \geq f(\alpha)$, where
\[
f(\alpha) = \begin{cases}4k+7, &\text{if $0 \leq \alpha \leq \frac{2}{3}$};\\[8pt]
\frac{k + 2}{1 - \alpha} + 3, &\text{if $\frac{2}{3} < \alpha < 1$}.
\end{cases}
\]
If $\rho_\alpha(G) \geq \rho_\alpha(K_{2k} \vee (K_{n - 2k - 1} \cup K_1))$, then $G$ is a fractional $k$-extendable graph unless $G = K_{2k} \vee (K_{n - 2k - 1} \cup K_1)$.
\end{theorem}

\begin{theorem}[\cite{MR4600143}]
Let $k$ be a positive integer and $n \geq 2k + 2$ be an even integer. If $G$ is a connected graph of order $n$, and
\[
\mu(G) \leq \mu(K_{2k} \vee (K_{n - 2k - 1} \cup K_{1})),
\]
then $G$ is $k$-extendable, unless $G \cong K_{2k} \vee (K_{n - 2k - 1} \cup K_{1})$.
\end{theorem}

\begin{theorem}[\cite{MR4899831}]
    Let $k$ be a positive integer and $n \geq 2k+3$ be an odd integer, and let $G$ be a connected graph of order $n$.

    \begin{enumerate}[label = (\roman*)]
        \item For $n \geq 2k+9$, if 
        \[
        \mu(G) \leq \mu\left(K_{2k} \vee (K_{n-2k-1} \cup K_{1})\right),
        \]
        then $G$ is a fractional $k$-extendable graph unless $G = K_{2k} \vee (K_{n-2k-1} \cup K_{1})$.

        \item For $n = 2k+3$, if 
        \[
        \mu(G) \leq \mu(K_{2k+1} \vee 2K_{1}),
        \]
        then $G$ is a fractional $k$-extendable graph unless $G = K_{2k+1} \vee 2K_{1}$.

        \item For $n = 2k+5$, if 
        \[
        \mu(G) \leq \mu(K_{2k+2} \vee 3K_{1}),
        \]
        then $G$ is a fractional $k$-extendable graph unless $G = K_{2k+2} \vee 3K_{1}$.

        \item For $n = 2k+7$, if 
        \[
        \mu(G) \leq \mu(K_{2k+3} \vee 4K_{1}),
        \]
        then $G$ is a fractional $k$-extendable graph unless $G = K_{2k+3} \vee 4K_{1}$.
    \end{enumerate}
\end{theorem}

Specifically, Zhou \cite{MR5004976} proved the following result.

\begin{theorem}[\cite{MR5004976}]
Let $k$ be a positive integer, and let $G$ be a connected graph of order $n$ with minimum degree $\delta$. 
\begin{enumerate}
    \item If $n \geq 2k + 9$ and
\[
\rho(G) \geq \rho(K_{2k} \vee (K_{n - 2k - 1} \cup K_1)),
\]
then $G$ is fractional $k$-extendable, unless $G \cong K_{2k} \vee (K_{n - 2k - 1} \cup K_1)$.
    \item If $n \geq 5\delta + 1$, $\delta \geq 2k + 1$ and
\[
\rho(G) \geq \rho(K_{\delta} \vee (K_{n - 2\delta + 2k - 1} \cup (\delta - 2k + 1)K_1)),
\]
then $G$ is fractional $k$-extendable, unless $G \cong K_{\delta} \vee (K_{n - 2\delta + 2k - 1} \cup (\delta - 2k + 1)K_1)$.
\end{enumerate}
\end{theorem}

Inspired by this work, we establish analogous sufficient conditions for fractional $k$-extendability based on the number of edges, the signless Laplacian spectral radius, and the distance spectral radius. Our main results are the following three theorems.

\begin{theorem}\label{thm:e-ext}
Let $k$ be a positive integer, and let $G$ be a connected graph of order $n$ with minimum degree $\delta$. 
\begin{enumerate}
    \item If $n \geq 2k + 9$ and
\[
e(G) \geq e(K_{2k} \vee (K_{n - 2k - 1} \cup K_1)),
\]
then $G$ is fractional $k$-extendable, unless $G \cong K_{2k} \vee (K_{n - 2k - 1} \cup K_1)$.
    \item If $n \geq 6\delta$, $\delta \geq 2k + 1$ and
\[
e(G) \geq e(K_{\delta} \vee (K_{n - 2\delta + 2k - 1} \cup (\delta - 2k + 1)K_1)),
\]
then $G$ is fractional $k$-extendable, unless $G \cong K_{\delta} \vee (K_{n - 2\delta + 2k - 1} \cup (\delta - 2k + 1)K_1)$.
\end{enumerate}
\end{theorem}

\begin{theorem}\label{thm:q-ext}
Let $k$ be a positive integer and let $G$ be a connected graph of order $n$ with minimum degree $\delta$. 
\begin{enumerate}
    \item If $n \geq 2k + 6$ and 
\[
q(G) \geq q(K_{2k} \vee (K_{n - 2k - 1} \cup K_1)),
\]
then $G$ is fractional $k$-extendable, unless $G \cong K_{2k} \vee (K_{n - 2k - 1} \cup K_1)$.
    \item If $n \geq 6.5\delta$, $\delta \geq 2k + 1$ and
\[
q(G) \geq q(K_{\delta} \vee (K_{n - 2\delta + 2k - 1} \cup (\delta - 2k + 1)K_1)),
\]
then $G$ is fractional $k$-extendable, unless $G \cong K_{\delta} \vee (K_{n - 2\delta + 2k - 1} \cup (\delta - 2k + 1)K_1)$.
\end{enumerate}
\end{theorem}

\begin{theorem}\label{thm:mu-ext}
Let $k$ be an integer, and let $G$ be a connected graph of order $n$ with minimum degree $\delta$. If $n \geq 12\delta - 2k + 1$, $\delta \geq 2k + 1$ and
\[
\mu(G) \leq \mu(K_{\delta} \vee (K_{n - 2\delta + 2k - 1} \cup (\delta - 2k + 1)K_1)),
\]
then $G$ is fractional $k$-extendable, unless $G \cong K_{\delta} \vee (K_{n - 2\delta + 2k - 1} \cup (\delta - 2k + 1)K_1)$.
\end{theorem}

\section{Preliminaries}

\begin{lemma}[\cite{MR2102027}]\label{lem:k-ext-condition}
Let $k$ be a positive integer, and let $G$ be a graph with a $k$-matching. Then $G$ is fractional $k$-extendable if and only if
\[
i(G - S) \leq |S| - 2k
\]
for every vertex set $S \subseteq V(G)$ with $G[S]$ containing a $k$-matching.
\end{lemma}

\begin{lemma}[\cite{MR3589612}]\label{lem:q-monotonicity}
Let $G$ be a connected graph and $xy \notin E(G)$. If $G + xy$ is connected and $H$ is a proper subgraph of $G$. We have,
\[
q(G + xy) > q(G) > q(H).
\]
\end{lemma}

\begin{lemma}[\cite{MR932967}]\label{lem:mu-monotonicity}
Let $e$ be an edge of a connected graph $G$. If $G - e$ is connected, then $\mu(G) < \mu(G - e)$.
\end{lemma}

Let $M$ be a real matrix of order $n$, and let $\mathcal{N} = \{1, 2, \dots, n\}$. Consider a partition $\pi: \{\mathcal{N}_1, \mathcal{N}_2 , \dots, \mathcal{N}_r\}$ of $\mathcal{N}$. With respect to $\pi$, the matrix $M$ can be written in block form as
\begin{align*}
M = \begin{pmatrix}
		M_{11} & M_{12} & \cdots & M_{1r}\\
		M_{21} & M_{22} & \cdots & M_{2r}\\
		\vdots & \vdots & \ddots & \vdots\\
		M_{r1} & M_{r2} & \cdots & M_{rr}\\
	\end{pmatrix},
\end{align*}
where $M_{ij}$ is the submatrix (block) formed by the rows in $\mathcal{N}_i$ and the columns in $\mathcal{N}_j$. Let $b_{ij}$ be the average sum of rows in block $M_{ij}$, \ie $b_{ij}$ be the sum of all entries in $M_{ij}$ divided by the number of rows in $M_{ij}$. The matrix $M_{\pi} = (b_{ij})$ is called the \emph{quotient matrix} of $M$ with respect to $\pi$. The partition $\pi$ is called \emph{equitable} if, for every pair $(i, j)$, each row of the block $M_{ij}$ has the same sum, i.e., $M_{ij}\mathbf{1} = b_{ij}\mathbf{1}$, where $\mathbf{1} = (1, 1, \dots, 1)^{\top}$. In this case, $M_{\pi}$ is called the \emph{equitable quotient matrix} of $M$.

By the Perron-Frobenius theorem \cite{MR2978290}, if $M$ is nonnegative and irreducible, then its spectral radius is a simple eigenvalue with an associated eigenvector $\mathbf{x} = (x_1, x_2, \dots, x_n)^{\top}$ that is unique up to scalar multiplication. This unique positive eigenvector, normalized to unit length, is the Perron vector of $M$.

\begin{lemma}[{\cite[Theorems 2.3 and 2.5]{MR3942724}}]\label{lem:equitable-quotient}
Let $M$ be a real matrix of order $n$ with an equitable partition $\pi$, and let $M_{\pi}$ be the corresponding equitable quotient matrix. Then the eigenvalues of $M_{\pi}$ are also eigenvalues of $M$. Furthermore, if $M$ is nonnegative, then the largest eigenvalues of $M$ and $M_{\pi}$ are equal.
\end{lemma}

For a connected graph $G$ on $n$ vertices, the \emph{Wiener index} is defined as $W(G) = \sum_{i<j} d_{ij}$, where $d_{ij}$ denotes the distance between vertices $i$ and $j$. The following bound follows directly from the Rayleigh quotient \cite{MR2978290}.

\begin{lemma}[\cite{MR2978290}]\label{lem:Wiener-bound}
Let $G$ be a connected graph of order $n$. Then
\[
\mu(G) = \max_{\mathbf{x} \neq \mathbf{0}} \frac{\mathbf{x}^{\top} \mathcal{D}(G)\mathbf{x}}{\mathbf{x}^{\top} \mathbf{x}} \geq \frac{\mathbf{1}^{\top} \mathcal{D}(G)\mathbf{1}}{\mathbf{1}^{\top} \mathbf{1}} = \frac{2W(G)}{n},
\]
where $\mathbf{1} = (1, 1, \dots, 1)^{\top}$ is the all-ones vector.
\end{lemma}

\section{Proof of \cref{thm:e-ext}}
Suppose that $G$ is not fractional $k$-extendable. By \cref{lem:k-ext-condition}, there exists a nonempty subset $S$ of $V(G)$ with $|S| \geq 2k$ such that $i(G - S) \geq |S| - 2k + 1$. Then $G$ is a spanning subgraph of $G_1 = K_s \vee (K_{n_1} \cup (s - 2k + 1)K_1)$, where $s \coloneqq |S| \geq 2k$ and $n_1 = n - 2s + 2k - 1 \geq 0$. We conclude that
\begin{align}\label{eq:3.1}
e(G) \leq e(G_1),
\end{align}
with equality if and only if $G \cong G_1$. Notice that $\delta = \delta(G) \leq \delta(G_1) = s$. 

Let $G_2 = K_{2k} \vee (K_{n - 2k - 1} \cup K_1)$. We show that $e(G_1) \leq e(G_2)$, with equality if and only if $G_1 \cong G_2$. If $s = 2k$, then $G_1 = G_2$ and $e(G_1) = e(G_2)$. Now, assume $s \geq 2k + 1$. A simple computation yields
\begin{align*}
e(G_2) - e(G_1) & = \binom{n - 1}{2} + 2k - \binom{n - s + 2k - 1}{2} - s(s - 2k + 1)\\
&=(s - 2k)n + 4ks - 2k^2 + 5k - \tfrac{3s^2}{2} - \tfrac{5s}{2}\\
&\eqqcolon w(n). 
\end{align*}
Then we proceed by the following cases.

\begin{case}
$n = 2s - 2k + 1$.
\end{case}

In this case, $w(n)$ becomes 
\[
w(n) = \tfrac{s^2}{2} + (-2k - \tfrac{3}{2})s + 2k^2 + 3k \eqqcolon h(s).
\]
Since $n = 2s - 2k + 1 \geq 2k + 9$, we have $s \geq 2k + 4$. The symmetry axis of $h(s)$ is $s_0 = 2k + \tfrac{3}{2} < 2k + 4$, so $h(s)$ is increasing for $s \geq 2k + 4$. Thus,
\begin{align*}
h(s) & \geq h(2k + 4) \\
    &= \tfrac{(2k + 4)^2}{2} + (-2k - \tfrac{3}{2})(2k + 4) + 2k^2 + 3k \\
    &= 2 > 0.
\end{align*}
Hence, $w(n) > 0$, implying $e(G_2) > e(G_1)$.

\begin{case}
$n \geq 2s - 2k + 2$.
\end{case}

\begin{subcase}
$s = 2k + 1$.
\end{subcase}

Then
\begin{align*}
w(n) & = (s - 2k)n + 4ks - 2k^2 + 5k - \tfrac{3s^2}{2} - \tfrac{5s}{2} \\
& = (2k + 1 - 2k)n + 4k(2k + 1) - 2k^2 + 5k - \tfrac{3(2k + 1)^2}{2} - \tfrac{5(2k + 1)}{2}\\
&= n - 2k - 4 > 0   \qquad \text{(since $n \geq 2k+9$)}
\end{align*}
Thus, $e(G_{2}) - e(G_{1}) = w(n) > 0$.
\begin{subcase}
$s \geq 2k + 2$.
\end{subcase}

Note that $w(n)$ is linear in $n$ with positive coefficient $s - 2k$, so it is increasing in $n$. For $n \geq 2s - 2k + 2$, 
\begin{align*}
w(n) & \geq w(2s - 2k + 2) \\
    &= (s - 2k)(2s - 2k + 2) + 4ks - 2k^2 + 5k - \tfrac{3s^2}{2} - \tfrac{5s}{2}\\
    &= \tfrac{s^2}{2} + (-2k - \tfrac{1}{2})s + 2k^2 + k \\
    &\eqqcolon h(s).
\end{align*}
The symmetry axis of $h(s)$ is $s_0 = 2k + \tfrac{1}{2} < 2k + 2$, so $h(s)$ is increasing for $s \geq 2k + 2$. Therefore,
\begin{align*}
h(s) &\geq h(2k + 2) \\
    &= \tfrac{(2k + 2)^2}{2} + (-2k - \tfrac{1}{2})(2k + 2) + 2k^2 + k\\
    &= 1 > 0.
\end{align*}
Thus, $w(n) \geq h(s) > 0$ and $e(G_2) > e(G_1)$.

In all cases, we have $e(G_1) < e(G_2)$ when $s \geq 2k + 1$. Combining with \eqref{eq:3.1}, we obtain  
\[
e(G) \leq e(G_1) < e(G_2) = e\left(K_{2k} \vee \left(K_{n - 2k - 1} \cup K_1\right)\right),
\]  
a contradiction.

In the following, we assume $\delta \geq 2k + 1$. Let $G_3 = K_{\delta} \vee (K_{n - 2\delta + 2k - 1} \cup (\delta - 2k + 1)K_1)$, where $n \geq 2\delta - 2k + 1$. We show that $e(G_1) \leq e(G_3)$, with equality if and only if $G_1 \cong G_3$. If $s = \delta$, then $G_1 = G_3$ and $e(G_1) = e(G_3)$. Now assume $s \geq \delta + 1$. Then
\begin{align*}
e(G_3) - e(G_1) & = \binom{n - \delta + 2k - 1}{2} + \delta(\delta - 2k + 1) - \binom{n - s + 2k - 1}{2} - s(s - 2k + 1) \\
& = (s - \delta)(n + 4k - \tfrac{3\delta}{2} - \tfrac{3s}{2} - \tfrac{5}{2}) \\
& = \tfrac{1}{2}(s - \delta)(2n + 8k - 3\delta - 3s - 5)
\end{align*}
Since $s \geq \delta + 1$, we have $s - \delta > 0$. Moreover, as $n \geq 6\delta$ and $n \geq 2s - 2k + 1$, it follows that  
\[
2n + 8k - 3\delta - 3s - 5 = \tfrac{3}{2}(n - 2s + 2k - 1) + \tfrac{1}{2}(n - 6\delta) + \tfrac{1}{2}(10k - 7) > 0.
\]  
Hence, $e(G_3) > e(G_1)$. Combining with \eqref{eq:3.1}, we get
\[
e(G) \leq e(G_1) < e(G_3) = e(K_{\delta} \vee (K_{n - 2\delta + 2k - 1} \cup (\delta - 2k + 1)K_1)),
\]
a contradiction.
\setcounter{case}{0}

\section{Proof of \cref{thm:q-ext}}
In this section, define $G_{1} \coloneqq K_s \vee (K_{n - 2s + 2k - 1} \cup (s - 2k + 1)K_1)$, and $G_{2} \coloneqq K_{2k} \vee (K_{n - 2k - 1}\cup K_1)$.

The equitable quotient matrix of $Q(G_2)$ with respect to the partition $V(G_2) = V(K_{2k}) \cup V(K_{n - 2k - 1}) \cup V(K_1)$ is
\[
B_2 = \begin{pmatrix}
         n + 2k - 2 & n - 2k - 1 & 1\\
            2k & 2n - 2k - 4 & 0\\
            2k & 0 & 2k\\
      \end{pmatrix},
\]
and its characteristic polynomial is 
\begin{equation}\label{CP-2}
f_2(x) = x^3 + (6 - 2k - 3n)x^2 + (6nk - 16k + 2n^2 - 8n + 8)x + (-4n^2 + 20n - 24)k.
\end{equation}
By \cref{lem:equitable-quotient}, $q(G_2)$ is the largest root of $f_2(x) = 0$.

$\bullet$ Assume $n \geq 2s - 2k + 2$. Let $\pi: V(G_1) \to (V(K_s), V(K_{n - 2s + 2k - 1}), V((s - 2k + 1)K_1))$. With respect to this partition $\pi$, the equitable quotient matrix of $Q(G_1)$ is
\[
B_{\pi, 1} = \begin{pmatrix}
    n + s - 2 & n - 2s + 2k - 1 & s - 2k + 1\\
            s & 2n - 3s + 4k - 4 & 0\\
            s & 0 & s\\
      \end{pmatrix},
\]
and its characteristic polynomial is 
\begin{align*}
f_{\pi, 1}(x) &= x^3 + (s - 3n - 4k + 6)x^2 + (-4s^2 + (8k + n - 4)s + 4kn - 8n - 8k + 2n^2 + 8)x\\
&\quad - 2s^3 + (8k + 4n - 10)s^2 + (- 8k^2 - 8kn + 20k - 2n^2 + 10n - 12)s.
\end{align*}
By \cref{lem:equitable-quotient}, $q(G_1)$ is the largest root of $f_{\pi, 1}(x) = 0$.

$\bullet$ Assume $n = 2s - 2k + 1$. Let $\pi': V(G_1) = (V(K_s), V((s - 2k)K_{1}), V(K_1))$.  With respect to this partition, the equitable quotient matrix of $Q(G_1)$ is
\[
B_{\pi', 1} = \begin{pmatrix}
  3s - 2k - 1 & s - 2k & 1\\
            s & s     & 0\\
            s & 0     & s\\
      \end{pmatrix},
\]
and its characteristic polynomial is
\[
f_{\pi', 1}(x) = x^3 + (2k - 5s + 1)x^2 + (6s^2 - 2ks - 3s)x - 2s^3 + 2s^2.
\]
By \cref{lem:equitable-quotient}, $q(G_1)$ is the largest root of $f_{\pi', 1}(x) = 0$.

\begin{lemma}\label{lem:q1q2}
Let $n, k, s$ be positive integers with $n \geq \max\{2s - 2k + 1,\, 2k + 6\}$ and $s \geq 2k$. Then
\[
q(K_s \vee (K_{n - 2s + 2k - 1} \cup (s - 2k + 1)K_1)) \leq q(K_{2k} \vee (K_{n - 2k - 1} \cup K_1)),
\]
where equality holds if and only if $s = 2k$.
\end{lemma}

\begin{proof}
Since $G_1 \cong G_2$ when $s = 2k$, we may assume $s \geq 2k + 1$.

\begin{case}
$n \geq 2s - 2k + 2$.
\end{case}

Since $K_{n - 1}$ is a proper subgraph of $G_2$, it follows from \cref{lem:q-monotonicity} that
\[
q(G_2) > q(K_{n - 1}) = 2n - 4.
\]
Next, we claim $f_{\pi, 1}(x) > f_2(x)$ for all $x > 2n - 4$. 

To compare $f_{\pi, 1}$ and $f_2$, define
\[
g(x) \coloneqq \frac{f_{\pi, 1}(x) - f_2(x)}{s - 2k}.
\]
A straightforward simplification gives
\begin{align*}
    g(x) = x^2 + (n - 4s - 4)x - 2n^2 - 12 + 10n - 10s + 4ns + 4ks - 2s^2.
\end{align*}
We show that $g(x) > 0$ for all $x > 2n - 4$. Since $n \geq 2s - 2k + 2 = (s - 2k - 1) + (s + 3) \geq s + 3$, the symmetry axis of $g(x)$ is to the left of $2n - 4$, hence $g(x)$ increases in $(2n - 4, + \infty)$.
Thus,
\[
g(x) > g(2n - 4) = 4n^2 - (4s + 18)n - 2s^2 + 6s + 4ks + 20 \eqqcolon h(n).
\]
Now, it suffices to prove $h(n) \geq 0$. 

If $s = 2k + 1$, the symmetry axis of $h(n)$ satisfies $\frac{4s + 18}{8} < 2k + 5 \leq n$, so
\begin{align*}
h(n) &= 4n^2 - (4s + 18)n - 2s^2 + 6s + 4ks + 20 \\
     &\geq 4(2k + 5)^2 - (4(2k + 1) + 18)(2k + 5) + 6(2k + 1) + 4k(2k + 1) - 2(2k + 1)^2 + 20 \\
     &= 4k + 14 > 0.
\end{align*}

If $s \geq 2k + 2$, then $\frac{4s + 18}{8} \leq 2s - 2k + 2 \leq n$ and
\begin{align*}
h(n) &= 4n^2 - (4s + 18)n + 6s + 4ks - 2s^2 + 20\\
     &\geq 4(2s - 2k + 2)^2 - (4s + 18)(2s - 2k + 2) + 6s + 4ks - 2s^2 + 20\\
     &= 6s^2 + (-20k - 6)s + 36k + 4(2k - 2)^2 - 16\\
     &\eqqcolon h_{1}(s).   
\end{align*}
The symmetry axis of $h_{1}(s)$ satisfies $\frac{20k + 6}{12} < 2k + 2$, so
\begin{align*}
    h_{1}(s)&= 6s^2 + (-20k - 6)s + 36k + 4(2k - 2)^2 - 16\\
    &\geq 6(2k + 2)^2 + (-20k - 6)(2k + 2) + 36k + 4(2k - 2)^2 - 16 \\
    &= 12 > 0.
\end{align*}    
Combining these results, we have $h(n) > 0$ for all $n \geq \max\{2s - 2k + 2,\, 2k + 5\}$, and thus $f_{\pi, 1}(x) > f_2(x)$ for all $x > 2n - 4$.

Since $q(G_2) > 2n - 4$, we have $f_{\pi, 1}(x) > f_2(x) \geq 0$ for all $x \geq q(G_2)$. Therefore, $f_{\pi, 1}(x) = 0$ has no roots in the interval $[q(G_2), + \infty)$, implying $q(G_1) < q(G_2)$.

\begin{case}
$n = 2s - 2k + 1$.
\end{case}
In this case, $G_{1} = K_{s} \vee (s - 2k + 1)K_{1}$ and $G_{2} = K_{2k} \vee (K_{2s - 4k} \cup K_{1})$. Substituting $n = 2s - 2k + 1$ into \eqref{CP-2}, we obtain
\begin{align*}
f_2(x) &= x^3 + (4k - 6s + 3)x^2 - (4k^2 + 4ks + 2k - 8s^2 + 8s - 2)x \\
       &\quad - 16ks^2 + 32k^2s - 24k^2 - 16k^3 - 8k + 24ks.
\end{align*}

Since $n = 2s - 2k + 1 \geq 2k + 6$, we have $s > 2k + 2$. Next, we claim $f_{\pi', 1}(x) > f_2(x)$ for all $x \in [ 4s - 4k - 2, + \infty )$. Define 
\[
w(x) \coloneqq \frac{f_{\pi', 1}(x) - f_2(x)}{s - 2k - 2}.
\]
Then
\[
w(x) = x^2 - (2s + 2k - 1)x - 2s^2 + (12k - 2)s - 8k^2 - 4k - 4 - \frac{8(k + 1)}{s - 2k - 2}.
\]
The symmetry axis $x = \frac{2s + 2k - 1}{2}$ is less than $4s - 4k - 2$, which implies that $w(x)$ is increasing on $[4s - 4k - 2, + \infty)$, additionally $\frac{10 + 20k}{12} < 2k + 6$.
Thus
\begin{align*}
w(x) &> w(4s - 4k - 2) \\
     &= 6s^2 - (20k + 10)s + 16k^2 + 12k - 2 - \tfrac{8(k + 1)}{s - 2k - 2}\\
     &\geq 6s^2 - (20k + 10)s + 16k^2 + 12k - 2 - 8(k + 1) \\
     &\geq 6(2k + 6)^2 - (20k + 10)(2k + 6) + 16k^2 + 12k - 2 - 8(k + 1) \\
     &= 8k + 146 > 0.
\end{align*}
Therefore, $f_{\pi', 1}(x) > f_2(x)$ for all $x \in [4s - 4k - 2, +\infty)$.

Since $q(G_2) > q(K_{2s - 2k}) = 4s - 4k - 2$, we have $f_{\pi', 1}(x) > f_2(x) \geq 0$ for all $x \geq q(G_2)$. Hence, $f_{\pi', 1}(x) = 0$ does not have roots in the interval $[q(G_2 ), + \infty)$, which implies $q(G_1) < q(G_2)$.

Combining both cases and the equality condition $s = 2k$, we conclude that $q(G_1) \leq q(G_2)$, with equality if and only if $G_1 \cong G_2$.
\setcounter{case}{0}
\end{proof}

\begin{lemma}\label{lem:q1q3}
Let $\delta, s, n, k$ be positive integers with $n \geq \max\{2s - 2k + 1, 6.5\delta\}$ and $s \geq \delta \geq 2k + 1$. Then
\[
q(K_s \vee (K_{n - 2s + 2k - 1} \cup (s - 2k + 1)K_1)) \leq q(K_{\delta} \vee (K_{n - 2\delta + 2k - 1} \cup (\delta - 2k + 1)K_1)),
\]
where equality holds if and only if the graphs are isomorphic.
\end{lemma}

\begin{proof}
Define $G_3 \coloneqq K_{\delta} \vee (K_{n - 2\delta + 2k - 1} \cup (\delta - 2k + 1)K_1)$. Given that $G_1$ is isomorphic to $G_3$ for $s = \delta$, it suffices to consider the case $s \geq \delta + 1$.

\begin{case}
$n \geq 2s - 2k + 2$.
\end{case}
The equitable quotient matrix of $Q(G_3)$ with respect to the partition $V(K_\delta) \cup V(K_{n - 2\delta + 2k - 1}) \cup V((\delta - 2k + 1)K_1)$ is
\[
B_3 = \begin{pmatrix}
          n + \delta - 2 & n - 2\delta + 2k - 1 & \delta - 2k + 1\\
            \delta & 2n - 3\delta + 4k - 4 & 0\\
            \delta & 0 & \delta\\
      \end{pmatrix},
\]
with characteristic polynomial
\begin{align*}
f_3(x) &= x^3 + (\delta - 3n - 4k + 6)x^2 + (-4\delta^2 + (8k + n - 4)\delta + 4kn - 8n - 8k + 2n^2 + 8)x \\
       &\quad - 2\delta^3 + (8k + 4n - 10)\delta^2 + (-8k^2 - 8kn + 20k - 2n^2 + 10n - 12)\delta.
\end{align*}
By \cref{lem:equitable-quotient}, $q(G_3)$ is the largest root of $f_3(x) = 0$.

Since $K_{n - \delta + 2k - 1}$ is a proper subgraph of $G_3$, it follows from \cref{lem:q-monotonicity} that
\[
q(G_3) > q(K_{n - \delta + 2k - 1}) = 2n - 2\delta + 4k - 4. 
\]
Define $h(x) \coloneqq \frac{f_{\pi, 1}(x) - f_3(x)}{s - \delta}$. We claim $h(x) > 0$ for all $x \geq 2n - 2\delta + 4k - 4$.
\begin{align*}
h(x) &= x^2 + (8k + n - 4 - 4\delta - 4s)x - 2(s^2 + s\delta + \delta^2) \\
     &\quad + (8k + 4n - 10)(s + \delta) - 8k^2 - 8kn + 20k - 2n^2 + 10n - 12.
\end{align*}
Since $2s \leq n + 2k - 1$ and $n \geq 6.5\delta$, the axis of symmetry of $h(x)$ satisfies
\[
\frac{4 + 4\delta + 4s - 8k - n}{2} < 2n - 2\delta + 4k - 4.
\]
Therefore,
\begin{align*}
h(x) &\geq h(2n - 2\delta + 4k - 4) \\
     &= 4n^2 + (28k - 14\delta - 4s - 18)n - 2s^2 + (6\delta + 6 - 8k)s \\
     &\quad + 40k^2 + (-40\delta - 60)k + 30\delta + 10\delta^2 + 20.
\end{align*}

First, suppose $s \geq \frac{16\delta}{5} - 1$. Since
\[
\frac{-28k + 14\delta + 4s + 18}{8} \leq 2s - 2k + 2,
\]
we have
\begin{align*}
h(2n - 2\delta + 4k - 4)
&\geq 4(2s - 2k + 2)^2 + (28k - 14\delta - 4s - 18)(2s - 2k + 2) - 2s^2 \\
&\quad + (6\delta + 6 - 8k)s + 40k^2 + (-40\delta - 60)k + 30\delta + 10\delta^2 + 20 \\
&= 6s^2 + (24k - 22\delta - 6)s + 2\delta - 12\delta k + 10\delta^2\\
&\eqqcolon z(s).
\end{align*}
The symmetry axis of $z(s)$ satisfies $\frac{6 + 22\delta - 24k}{12} < \frac{16\delta}{5} - 1$, so
\begin{align*}
    z(s)&= 6s^2 + (24k - 22\delta - 6)s + 2\delta - 12\delta k + 10\delta^2\\
    &\geq 6\left(\tfrac{16\delta}{5} - 1\right)^2 + (24k - 22\delta - 6)\left(\tfrac{16\delta}{5} - 1\right) + 2\delta - 12\delta k + 10\delta^2\\ 
    &= \tfrac{26}{25}\delta^2 + (\tfrac{324}{5}k - \tfrac{168}{5})\delta - 24k + 12 \\
    &\geq \tfrac{26}{25}(2k + 1)^2 + (\tfrac{324}{5}k -\tfrac{168}{5})(2k + 1) - 24k + 12 \\
    &= \tfrac{2}{25}(1672k^2 - 278k - 257) > 0.
    \end{align*}
Now suppose $\delta + 1 \leq s < \frac{16\delta}{5} - 1$. Then
\[
\frac{6\delta + 6 - 8k - 4n}{4} \leq \frac{(\delta + 1) + (\frac{16}{5}\delta - 1)}{2},
\]
and
\begin{align*}
&h(2n - 2\delta + 4k - 4) \\
&= -2s^2 + (6\delta + 6 - 8k - 4n)s + 4n^2 + (28k - 14\delta - 18)n \\
&\quad + 40k^2 + (-40\delta - 60)k + 30\delta + 10\delta^2 + 20 \\
&\geq -2(\tfrac{16\delta}{5} - 1)^2 + (6\delta + 6 - 8k - 4n)(\tfrac{16\delta}{5} - 1) + 4n^2 + (28k - 14\delta - 18)n \\
&\quad + 40k^2 + (-40\delta - 60)k + 30\delta + 10\delta^2 + 20 \\
&= 4n^2 + \left(28k - 14\delta - 4(\tfrac{16}{5}\delta - 1) - 18\right)n - 2(\tfrac{16}{5}\delta - 1)^2 \\
&\quad + (6\delta + 6 - 8k)(\tfrac{16}{5}\delta - 1) + 40k^2 + (-40\delta - 60)k + 30\delta + 10\delta^2 + 20 \\
&\eqqcolon r(s).
\end{align*}
The symmetry axis of $r(s)$ satisfies $\frac{28k - 14\delta - 4(\frac{16}{5}\delta - 1) - 18}{-8} < 6.5\delta$, so
\begin{align*}
r(s)&= 4n^2 + \left(28k - 14\delta - 4(\tfrac{16}{5}\delta - 1) - 18\right)n - 2(\tfrac{16}{5}\delta - 1)^2 \\
&\quad + (6\delta + 6 - 8k)(\tfrac{16}{5}\delta - 1) + 40k^2 + (-40\delta - 60)k + 30\delta + 10\delta^2 + 20 \\
&\geq 4(6.5\delta)^2 + \left(28k - 14\delta - 4(\tfrac{16}{5}\delta - 1) - 18\right)(6.5\delta) - 2(\tfrac{16}{5}\delta - 1)^2 \\
&\quad + (6\delta + 6 - 8k)(\tfrac{16}{5}\delta - 1) + 40k^2 + (-40\delta - 60)k + 30\delta + 10\delta^2 + 20 \\
&= 3.52\delta^2 + (116.4k - 35)\delta + 40k^2 - 52k + 12 \\
&> 40k^2 - 52k + 12 \geq 0.
\end{align*}
Therefore, $h(x) > 0$ for all $x \geq 2n - 2\delta + 4k - 4$, i.e., $f_{\pi, 1}(x) > f_3(x)$ for all $x \geq 2n - 2\delta + 4k - 4$.

Since $q(G_3) > 2n - 2\delta + 4k - 4$, it follows that $f_{\pi, 1}(x) > f_3(x) \geq 0$ for all $x \geq q(G_3)$. Consequently, $f_{\pi, 1}(x) = 0$ has no roots in $[q(G_3), +\infty)$, yielding $q(G_1) < q(G_3)$.
\begin{case}
$n = 2s - 2k + 1$.
\end{case}
In this case, $n = 2s - 2k +1 \geq 6.5\delta$. Then $G_{1} = K_{s} \vee (s - 2k + 1)K_{1}$ and $G_{3} = K_{\delta} \vee (K_{2s - 2\delta} \cup (\delta - 2k + 1)K_{1})$.

The equitable quotient matrix of $Q(G_2)$ with respect to the partition $V(G_3) = V(K_{\delta}) \cup V(K_{2s - 2\delta}) \cup (\delta - 2k + 1) V(K_{1})$ is
\[
B_3 = \begin{pmatrix}
        2s - 2k+\delta - 1 & 2s - 2\delta &  \delta  - 2k + 1\\
            \delta & 4s - 3\delta & 0\\
            \delta & 0 & \delta\\
      \end{pmatrix},
\]
with characteristic polynomial
\begin{align*}
f_3(x) &=x^3 + (\delta + 2k - 6s + 1)x^2 + (\delta - 4s + 6\delta k + 2\delta s - 8ks - 4\delta^2 + 8s^2)x\\
& - 2\delta^3 + 8\delta^2 s - 6\delta^2 - 8\delta s^2 + 8\delta s
\end{align*}
To compare $f_{\pi', 1}$ and $f_3$, define
\[
w(x) \coloneqq \frac{f_{\pi', 1}(x) - f_3(x)}{s - \delta}.
\]
we have
\begin{align*}
    w(x) = x^2 - (2s + 4\delta - 6k -1)x + 2s^2 + 2\delta^2 - 6s\delta -2s + 6\delta.
\end{align*}
The symmetry axis $x = \frac{2s + 4\delta - 6k -1}{2}$ is less than $4s - 2\delta - 2$, which implies that $w(x)$ is increasing on $[4s - 2\delta - 2, + \infty)$, in addition, we have $\tfrac{-24k + 34\delta + 10}{20} < \tfrac{6.5\delta + 2k - 1}{2} \leq s$.
Thus,
\begin{align*}
w(x) &> w(4s - 2\delta - 2) \\
     &=10s^2 + (24k - 34\delta - 10)s +20\delta - 12k - 12\delta k + 14\delta^2 + 2\\
     & \geq 10(\tfrac{6.5\delta + 2k - 1}{2})^2 + (- 22\delta - 10)(\tfrac{6.5\delta + 2k - 1}{2})+20\delta - 12k - 12\delta k + 14\delta^2 + 2\\
     &=\tfrac{73\delta^2}{8} + (98k - 28)\delta + 34k^2 - 44k + \tfrac{19}{2}>0
\end{align*}
Therefore, $f_{\pi', 1}(x) > f_3(x)$ for all $x \in [4s - 2\delta - 2, +\infty)$.

Since $q(G_3) > q(K_{2s - \delta}) = 4s - 2\delta - 2$, we have $f_{\pi', 1}(x) > f_3(x) \geq 0$ for all $x \geq q(G_3)$. Hence, $f_{\pi', 1}(x) = 0$ does not have roots in the interval $[q(G_3 ), + \infty)$, which implies $q(G_1) < q(G_3)$.

Together with the case $s = \delta$, we conclude that $q(G_1) \leq q(G_3)$, with equality if and only if $G_1 \cong G_3$.
\end{proof}

\begin{proof}[Proof of \cref{thm:q-ext}]
Suppose that $G$ is not fractionally $k$-extendable. By \cref{lem:k-ext-condition}, there exists a nonempty subset $S \subset V(G)$ with $|S| = s \geq 2k$ such that $i(G - S) \geq s - 2k + 1$. Then $n \geq i(G - S) + |S| \geq 2s - 2k + 1$, and $G$ is a spanning subgraph of $K_s \vee (K_{n - 2s + 2k - 1} \cup (s - 2k + 1)K_1)$. Let $\delta = \delta(G) \leq \delta(K_s \vee (K_{n - 2s + 2k - 1} \cup (s - 2k + 1)K_1)) = s$. By \cref{lem:q-monotonicity}, we conclude
\[
q(G) \leq q(K_s \vee (K_{n - 2s + 2k - 1} \cup (s - 2k + 1)K_1))
\]
with equality if and only if $G \cong K_s \vee (K_{n - 2s + 2k - 1} \cup (s - 2k + 1)K_1)$.

By \cref{lem:q1q2},
\[
q(G) \leq q(K_s \vee (K_{n - 2s + 2k - 1} \cup (s - 2k + 1)K_1)) \leq q(K_{2k} \vee (K_{n - 2k - 1}\cup K_1)),
\]
with equality if and only if $G \cong K_{2k} \vee (K_{n - 2k - 1} \cup K_1)$, a contradiction.

If $\delta \geq 2k + 1$, then by \cref{lem:q1q3},
\[
q(G) \leq q(K_s \vee (K_{n - 2s + 2k - 1} \cup (s - 2k + 1)K_1)) \leq q(K_{\delta} \vee (K_{n - 2\delta + 2k - 1} \cup (\delta - 2k + 1)K_1))
\]
with equality if and only if $G \cong K_{\delta} \vee (K_{n - 2\delta + 2k - 1} \cup (\delta - 2k + 1)K_1)$, again a contradiction.

This completes the proof.
\setcounter{case}{0}
\end{proof}

\section{Proof of \cref{thm:mu-ext}}
Let $G$ be a connected graph of order $n$ and minimum degree $\delta$. Suppose that $G$ is not fractional $k$-extendable. By \cref{lem:k-ext-condition}, there exists a nonempty subset $S \subset V(G)$ with $|S| \geq 2k$ such that $i(G - S) \geq |S| - 2k + 1$. Then $G$ is a spanning subgraph of $G_1 \coloneqq K_s \vee (K_{n_1} \cup (s - 2k + 1)K_1)$, where $s \coloneqq |S| \geq 2k$ and $n_1 \coloneqq n - 2s + 2k - 1 \geq 0$. Using \cref{lem:mu-monotonicity}, we conclude
\begin{align}\label{eq:5.1}
\mu(G) \geq \mu(G_1),
\end{align}
with equality if and only if $G \cong G_1$. Let $\delta \coloneqq \delta(G)$. Since $s = \delta(G_1) \geq \delta(G) = \delta$, we proceed by cases.

Let $G_3 \coloneqq K_{\delta} \vee (K_{n - 2\delta + 2k - 1} \cup (\delta - 2k + 1)K_1)$, where $n \geq 2\delta - 2k + 1$. If $s = \delta$, then $G_1 \cong G_3$, and so $\mu(G_1) = \mu(G_3)$. In what follows, we consider the case where $s \geq \delta + 1$. 

\begin{case}
$n \geq 2s - 2k + 2$.
\end{case}
Consider the partition $V(G_1) = V(K_s) \cup V(K_{n - 2s + 2k - 1}) \cup V((s - 2k + 1)K_1)$, the equitable quotient matrix of $\mathcal{D}(G_{1})$ is
\[
B_1 =
\begin{pmatrix}
    s - 1 & n - 2s + 2k - 1 & s - 2k + 1 \\
    s & n - 2s + 2k - 2 & 2(s - 2k + 1) \\
    s & 2(n - 2s + 2k - 1) & 2(s - 2k)
\end{pmatrix}.
\]
The characteristic polynomial of $B_1$ is
\begin{align*}
\varphi_{B_1}(x) &= x^3 + (2k - n - s + 3)x^2 \\
&\quad + (5s^2 - 14ks - 2ns + 8k^2 + 4kn + 6s - 6k - 5n + 6)x \\
&\quad + (-2s^3 + 6ks^2 + ns^2 + 2s^2- 4k^2 s - 2kns - 10ks - ns + 6s + 8k^2 + 4kn - 8k - 4n + 4).
\end{align*}
By \cref{lem:equitable-quotient}, $\mu(G_1)$ equals the largest root of $\varphi_{B_1}(x) = 0$. 

Consider the partition $V(G_3) = V(K_\delta) \cup V(K_{n - 2\delta + 2k - 1}) \cup V((\delta - 2k + 1)K_1)$. The equitable quotient matrix of $\mathcal{D}(G_3)$ is
\[
B_3 =
\begin{pmatrix}
    \delta - 1 & n - 2\delta + 2k - 1 & \delta - 2k + 1 \\
    \delta & n - 2\delta + 2k - 2 & 2(\delta - 2k + 1) \\
    \delta & 2(n - 2\delta + 2k - 1) & 2(\delta - 2k)
\end{pmatrix}.
\]
Let $\varphi_{B_3}(x)$ denote the characteristic polynomial of $B_3$:
\begin{align*}
\varphi_{B_3}(x) &= x^3 + (2k - n - \delta + 3)x^2 \\
&\quad + (5\delta^2 - 14k\delta - 2n\delta + 8k^2 + 4kn + 6\delta - 6k - 5n + 6)x \\
&\quad + (-2\delta^3  + 6k\delta^2 + n\delta^2 + 2\delta^2 - 4k^2\delta - 2kn\delta - 10k\delta - n\delta + 6\delta + 8k^2 + 4kn - 8k - 4n + 4).
\end{align*}
By \cref{lem:equitable-quotient}, $\mu(G_3)$ equals the largest root of $\varphi_{B_3}(x) = 0$. By \cref{lem:Wiener-bound},
\begin{align*}
\mu(G_3) & \geq \frac{2W(G_{3})}{n} \\
         & = \frac{2}{n}\left(\binom{n - \delta + 2k - 1}{2} + 2\binom{\delta - 2k + 1}{2} + \delta(\delta - 2k + 1) + 2(n - 2\delta + 2k - 1)(\delta - 2k + 1)\right) \\
         & = n - \delta + 2k  - 2 +  \frac{(\delta - 2k + 1)(3n - 3\delta + 2k - 2)}{n} \\
         & \geq n - \delta + 2k  - 2 + \frac{2(3n - 0.5n)}{n} \qquad (\text{since $n \geq 6\delta$, $\delta \geq 2k + 1$ and $k \geq 1$})\\
         & = n - \delta + 2k + 3.
\end{align*}
Define $\varphi_{B_3}(x) - \varphi_{B_1}(x) = (s - \delta)h(x)$. A straightforward computation yields
\begin{align*}
h(x) &= x^2 + (14k - 5\delta + 2n - 5s - 6)x \\
&\quad + (2\delta^2 + 2\delta s + 2s^2 - (6k + n + 2)\delta - (6k + n + 2)s + 4k^2 + 2kn + 10k + n - 6).
\end{align*}
The axis of symmetry of $h(x)$ is $x_0 = -\frac{14k - 5\delta + 2n - 5s - 6}{2}$. Note that
\[
x_0 \leq n - \delta + 2k + 3 \quad \Leftrightarrow \quad n \geq \frac{7\delta - 18k + 5s}{4} = \frac{2.5(2s - 2k + 2)}{4} + \frac{7\delta - 13k - 5}{4}.
\]
Consequently, conditions $n \geq 5\delta$ and $n \geq 2s - 2k + 2$ are sufficient. Thus, $h(x)$ is increasing for $x \geq n - \delta + 2k + 3$. Consequently,
\begin{align*}
h(x) &\geq h(n - \delta + 2k + 3) \\
&= 3n^2 + (24k - 10\delta - 6s + 7)n \\
&\quad + 8\delta^2 - 34\delta k + 7\delta s - 17\delta + 36k^2 - 16ks + 52k + 2s^2 - 17s - 15 \\
&\eqqcolon w(n).
\end{align*}

\begin{subcase}
$ \delta + 1 \leq s \leq 6\delta - 10k$.
\end{subcase}

For $\delta + 1 \leq s \leq 6\delta - 10k$, we bound $w(n)$ from below:
\begin{align*}
w(n) &= 3n^2 + (24k - 10\delta - 6s + 7)n \\
&\quad + 8\delta^2 - 34\delta k + 7\delta s - 17\delta + 36k^2 - 16ks + 52k + 2s^2 - 17s - 15 \\
&= 2s^2 + (-6n + 7\delta - 16k - 17)s + (3n^2 + (24k - 10\delta + 7)n + 8\delta^2 - 34\delta k - 17\delta + 36k^2 + 52k - 15) \\
&\geq 2(6\delta - 10k)^2 + (-6n + 7\delta - 16k - 17)(6\delta - 10k) \\
&\quad + 3n^2 + (24k - 10\delta + 7)n + 8\delta^2 - 34\delta k - 17\delta + 36k^2 + 52k - 15 \\
&= 3n^2 - (46\delta - 84k - 7)n + 122\delta^2 -440\delta k - 119\delta + 396k^2 + 222k - 15 \\
&\eqqcolon h(n).
\end{align*}
The axis of symmetry of $h(n)$ is $n_0 = \frac{46\delta - 84k - 7}{6}$. Since $n \geq 12\delta - 2k + 1$, we have $n_0 \leq 12\delta - 2k + 1$, so $h(n)$ is increasing for $n \geq 12\delta - 2k + 1$. Thus,
\begin{align*}
h(n) &\geq h(12\delta - 2k + 1) \\
&= 3(12\delta - 2k + 1)^2 - (46\delta - 84k - 7)(12\delta - 2k + 1) + 122\delta^2 -440\delta k - 119\delta + 396k^2 + 222k - 15 \\
&= 2\delta^2 + (516k - 9)\delta + 240k^2 + 280k - 5 \\
&> 0.
\end{align*}

\begin{subcase}
$s \geq  6\delta - 10k + 1$.
\end{subcase}

Recall that
\[
w(n) = 3n^2 + (24k - 10\delta - 6s + 7)n + 8\delta^2 - 34\delta k + 7\delta s - 17\delta + 36k^2 - 16ks + 52k + 2s^2 - 17s - 15.
\]
The axis of symmetry of $w(n)$ is $n_0 = -\frac{24k - 10\delta - 6s + 7}{6}$. Given $s \geq  6\delta - 10k + 1  \geq \frac{10\delta - 12k - 19}{6}$, we have $n_0 \leq 2s - 2k + 2$, so $w(n)$ is increasing for $n \geq 2s - 2k + 2$. Thus,
\begin{align*}
w(n) &\geq w(2s - 2k + 2) \\
&= 2s^2 + (20k - 13\delta + 9)s + 8\delta^2 - 14\delta k + 62k - 37\delta + 11 \\
&\eqqcolon g(s).
\end{align*}
The axis of symmetry of $g(s)$ is $s_0 = -\frac{20k - 13\delta + 9}{4} \leq 6\delta - 10k + 1$, so
\begin{align*}
g(s) &\geq g(6\delta - 10k + 1) \\
&= 2(6\delta - 10k + 1)^2 + (20k - 13\delta + 9)(6\delta - 10k + 1) + 8\delta^2 - 14\delta k + 62k - 37\delta + 11 \\
&= 2\delta^2 - (4k - 28)\delta - 48k + 22 \\
&\geq 2(2k + 1)^2 - (4k - 28)(2k + 1) - 48k + 22 \quad\text{(since $\delta \geq 2k + 1$)}\\
&= 12k + 52 > 0.
\end{align*}

\begin{case}
$n = 2s - 2k + 1$.
\end{case}

The graph $G_3$ becomes
\[
G_3 = K_{\delta} \vee \left( K_{2s - 2\delta} \cup (\delta - 2k + 1)K_1 \right).
\]
With respect to the partition 
\[
V(G_3) = V(K_\delta) \cup V(K_{2s-2\delta}) \cup V((\delta - 2k + 1)K_1),
\]
the equitable quotient matrix of $\mathcal{D}(G_{3})$ is
\[
B_3 =
\begin{pmatrix}
    \delta - 1 & 2s - 2\delta & \delta - 2k + 1 \\
    \delta & 2s - 2\delta - 1 & 2(\delta - 2k + 1) \\
    \delta & 2(2s - 2\delta) & 2(\delta - 2k)
\end{pmatrix}.
\]
Let $\varphi_{B_3}(x)$ denote the characteristic polynomial of $B_3$:
\begin{align*}
\varphi_{B_3}(x) &= x^3 + (4k - \delta - 2s + 2)x^2 \\
&\quad + (4\delta + 8k - 10s - 10\delta k - 4\delta s + 8ks + 5\delta^2 + 1)x \\
&\quad + (5\delta + 4k - 8s - 10\delta k - 2\delta s + 8ks + 4\delta^2k + 2\delta^2s + 3\delta^2 - 2\delta^3 - 4\delta ks).
\end{align*}
For $n = 2s - 2k + 1$,
\begin{equation*}
\varphi_{B_1}(x) = x^3 + (4k - 3s + 3)x^2 + (12k - 9s - 2ks + s^2 + 2)x + 8k - 6s - 4ks + 2s^2.
\end{equation*}
Define $f(x) \coloneqq \varphi_{B_3}(x) - \varphi_{B_1}(x)$. A straightforward calculation yields the following:
\begin{align*}
f(x) & =(s - \delta - 1)x^2 + (4\delta - 4k - s - 10\delta k - 4\delta s + 10ks + 5\delta^2 - s^2 - 1)x\\
&+ 5\delta - 4k - 2s - 10\delta k - 2\delta s + 12ks + 4\delta^2k + 2\delta^2s + 3\delta^2 - 2\delta^3 - 2s^2 - 4\delta ks.
\end{align*}
The symmetry axis of $f(x)$ is
\[
x_0 = \tfrac{4\delta - 4k - s - 10\delta k - 4\delta s + 10ks + 5\delta^2 - s^2 - 1}{-2(s - \delta - 1)}.
\]
Since $n = 2s - 2k + 1$ and $n \geq 12\delta - 2k + 1$, we have $s \geq 6\delta$.

We claim that $x_0 \leq 2s - \delta + 4$.
\begin{proof}[Proof of the Claim]
We aim to prove $x_0 \leq 2s - \delta + 4$, which is equivalent to
\[
4\delta - 4k - s - 10\delta k - 4\delta s + 10ks + 5\delta^2 - s^2 - 1 \geq (2s - \delta + 4)(-2s + 2\delta + 2).
\]
Simplifying the right-hand side:
\[
(2s - \delta + 4)(-2s + 2\delta + 2) = -4s^2 + (6\delta - 4)s - 2\delta^2 + 6\delta + 8,
\]
we obtain the equivalent inequality:
\[
3s^2 - (10\delta - 10k - 3)s - 10k\delta + 7\delta^2 - 2\delta - 4k - 9 \geq 0.
\]
Define the quadratic function:
\[
h(s) = 3s^2 - (10\delta - 10k - 3)s - 10k\delta + 7\delta^2 - 2\delta - 4k - 9.
\]
The axis of symmetry of $ h(s) $ is
\[
s_0 = \tfrac{10\delta - 10k - 3}{6}.
\]
Thus,
\[
s_0 = \tfrac{10\delta - 10k - 3}{6} \leq 6\delta,
\]
as verified by:
\[
\tfrac{10\delta - 10k - 3}{6} \leq 6\delta \quad \Leftrightarrow \quad 10\delta - 10k - 3 \leq 36\delta \quad \Leftrightarrow \quad 26\delta + 10k + 3 \geq 0.
\]
Therefore, for $s \geq 6\delta$, the function $ h(s) $ is increasing, and we have
\begin{align*}
h(s) &\geq h(6\delta) \\
     &= 3(6\delta)^2 - (10\delta - 10k - 3)(6\delta) - 10k\delta + 7\delta^2 - 2\delta - 4k - 9 \\
     &= 55\delta^2 + (50k + 16)\delta - 4k - 9 \\
     &= 55\delta^2 + (50k\delta - 4k) + (16\delta - 9) > 0.
\end{align*}
This completes the proof of the claim.
\end{proof}
Now consider $f(x)$ for $ x \geq 2s - \delta + 4$. Since $ f(x)$ is a quadratic function with positive leading coefficient $ (s - \delta - 1 > 0) $ and axis of symmetry $ x_0 \leq 2s - \delta + 4$, it is increasing on $ [2s - \delta + 4, \infty) $. Therefore,
\begin{align*}
f(x) &\geq f(2s - \delta + 4) \\
     &= 2s^3 + (20k - 15\delta + 4)s^2 + (21\delta^2 + 44k - 34\delta k - 29\delta - 8)s \\
     &\quad - 8\delta^3 + 26\delta^2 + 14\delta - 20k - 46\delta k + 14\delta^2 k - 20 \\
     &\eqqcolon g(s).
\end{align*}
We now show that $g(s) > 0$ for $s \geq 6\delta$. Differentiating $g(s)$ with respect to $s$:
\[
g'(s) = 6s^2 + (40k - 30\delta + 8)s + 21\delta^2 + 44k - 34\delta k - 29\delta - 8.
\]
The axis of symmetry of $g'(s)$ is
\[
s_0' = \tfrac{-40k + 30\delta - 8}{12}.
\]
Thus,
\[
s_0' = \tfrac{-40k + 30\delta - 8}{12} \leq 6\delta,
\]
as verified by:
\[
\tfrac{-40k + 30\delta - 8}{12} \leq 6\delta \quad \Leftrightarrow \quad -40k + 30\delta - 8 \leq 72\delta \quad \Leftrightarrow \quad  42\delta + 40k + 8 \geq 0.
\]
Therefore, for $s \geq 6\delta$, the function $g'(s)$ is increasing, and we have
\begin{align*}
g'(s) &\geq g'(6\delta) \\
      &= 6(6\delta)^2 + (40k - 30\delta + 8)(6\delta) + 21\delta^2 + 44k - 34\delta k - 29\delta - 8 \\
      &= 57\delta^2 + (206k + 19)\delta + 44k - 8 > 0.
\end{align*}
Hence, $ g(s) $ is strictly increasing for $s \geq 6\delta$, and
\begin{align*}
g(s) &\geq g(6\delta) \\
     &= 2(6\delta)^3 + (20k - 15\delta + 4)(6\delta)^2 + (21\delta^2 + 44k - 34\delta k - 29\delta - 8)(6\delta) \\
     &\quad - 8\delta^3 + 26\delta^2 + 14\delta - 20k - 46\delta k + 14\delta^2 k - 20 \\
     &= (10\delta^3 - 20k - 20) + (530k - 4)\delta^2 + (218k - 34)\delta \\
     &> 0 \quad \text{(since $\delta \geq 2k + 1$)}.
\end{align*}
Therefore, $f(x) \geq g(s) > 0$ for all $x \geq 2s - \delta + 4$, which implies $\varphi_{B_3}(x) > \varphi_{B_1}(x)$ for all $x \geq 2s - \delta + 4$. Since $\mu(G_3) \geq n - \delta + 2k + 3 = 2s - \delta + 4$, we conclude that $\mu(G_1) > \mu(G_3)$.

In all cases, we have $h(x) > 0$ when $x \geq n - \delta + 2k + 3$. Hence, $\varphi_{B_3}(x) - \varphi_{B_1}(x) > 0$ for $x \geq n - \delta + 2k + 3$, implying $\mu(G_1) > \mu(G_3)$. Combined with the inequalities \eqref{eq:5.1}, this yields $\mu(G) > \mu(G_3) = \mu(K_\delta \vee (K_{n - 2\delta + 2k - 1} \cup (\delta - 2k + 1)K_1))$, a contradiction.

\end{document}